\documentclass[11pt, a4paper]{amsart}

\usepackage{amsfonts,amsmath,amssymb, amscd,fullpage}
\usepackage{stmaryrd}
\usepackage[all]{xy}

\newtheorem{theorem}{Theorem}[section]
\newtheorem{lemma}[theorem]{Lemma}
\newtheorem{definition}[theorem]{Definition}
\newtheorem{proposition}[theorem]{Proposition}
\newtheorem{corollary}[theorem]{Corollary}
\newtheorem{example}[theorem]{Example}
\newtheorem{remark}[theorem]{Remark}

\theoremstyle{definition}

\newcommand\pf{\begin{proof}}
\newcommand\epf{\end{proof}}

\DeclareMathOperator{\reg}{reg}

\numberwithin{equation}{section}

\title{Half-liberated real spheres and their subspaces}

\author{Julien Bichon}
\address{
Laboratoire de Math\'ematiques,
Universit\'e Blaise Pascal, Campus des c\'ezeaux, 3 place Vasarely,
63178 Aubi\`ere Cedex, France}
\email{Julien.Bichon@math.univ-bpclermont.fr}


\begin{document}

\begin{abstract}
 We describe the quantum subspaces of Banica-Goswami's half-liberated real-spheres, showing in particular that there is a bijection between the symmetric ones and the conjugation stable closed subspaces of the complex projective spaces. 
\end{abstract}

\maketitle

\section{Introduction}

Let $n \geq 1$. The half-liberated real sphere $S_{\mathbb R, *}^{n-1}$ was defined by Banica and Goswami \cite{bago} as the quantum space corresponding to the $C^*$-algebra
$$C(S_{\mathbb R, *}^{n-1})= C^*\left(v_1, \ldots , v_n \ | \ \sum_{i=1}^nv_i^2=1, \ v_i^*=v_i, \ v_iv_jv_k=v_kv_jv_i, \ 1 \leq i,j,k\leq n\right)$$
It corresponds to a natural quantum homogeneous space over the half-liberated  orthogonal quantum group $O_n^*$ introduced by Banica and Speicher in \cite{basp}. These quantum spaces and groups, although defined by very simple means by the intriguing half-commutativity relations $abc=cba$ arising from representation-theoretic considerations via Woronowicz' Tannaka-Krein duality \cite{wor}, turned out to be new in the field, and definitively of interest. See \cite{ban1,ban2} for recent developments and general discussions on non-commutative spheres.

The aim of this paper is to describe the  quantum subspaces of $S_{\mathbb R, *}^{n-1}$: we will show in particular that there is a natural bijection between the symmetric ones (see Section \ref{sec:sym} for the definition) and the conjugation stable closed subspaces of the complex projective space $P^{n-1}_{\mathbb C}$.
The description of all subspaces is also done, but is more technical, and uses representation theory methods, inspired by those used by Podle\'s \cite{pod} in the determination of the quantum subgroups of $SU_{-1}(2)\simeq O_2^*$.
It follows from our analysis that the quantum subspaces of $S_{\mathbb R, *}^{n-1}$ completely can be described by means of classical spaces, a fact already noted in the description of the quantum subgroups of $O_n^*$ in \cite{bdv}.  
As in \cite{bdv}, a crossed product model  provides the bridge linking the quantum subspaces and the subspaces of an appropriate classical space.

In fact, for the description of all the quantum subspaces, we will work in a more general context, where a quantum space $Z_{\mathbb R, *}$ is associated to any compact space $Z$ endowed with an appropriate  $\mathbb T \rtimes \mathbb Z_2 \simeq O_2$ continuous action, and where $S_{\mathbb R, *}^{n-1}$ is obtained from the complex sphere $S^{n-1}_{\mathbb C}$. We do not know if our general framework really furnishes new examples of interest, however we think that, as often with abstract settings,  it has the merit to clean-up arguments, can ultimately simplify the theory, and could be well-suited for other developments, such as $K$-theory computations.

The paper is organized as follows. Section 2 consists of preliminaries. In Section 3, we construct a faithful crossed product representation of $C(S_{\mathbb R, *}^{n-1})$ and use the sign  automorphism to construct a $\mathbb Z_2$-grading on   $C(S_{\mathbb R, *}^{n-1})$. We then show that there is an explicit bijective correspondence between symmetric quantum subspaces of $S_{\mathbb R, *}^{n-1}$ (corresponding to $\mathbb Z_2$-graded ideals of $C(S_{\mathbb R, *}^{n-1})$) and 
conjugation stable closed subspaces of the complex projective space $P^{n-1}_{\mathbb C}$. Section 4, in which we work in a slightly more general framework, is devoted to the description of all the quantum subspaces of $S_{\mathbb R, *}^{n-1}$, in terms of pairs of certain subspaces of the complex sphere $S^{n-1}_{\mathbb C}$.

\smallskip

\textbf{Acknowledgements.} I wish to thank Teodor Banica for useful discussions.

\section{Notations and preliminaries}\label{se:prelim}

\subsection{Conventions}
All $C^*$-algebras are assumed to be unital, as well as all $C^*$-algebra maps. By a compact space we mean a compact Hausdorff space,  unless otherwise specified. The categories of compact spaces and of commutative $C^*$-algebras are anti-equivalent by the classical Gelfand duality, and we use, in a standard way, the language of  quantum spaces: a $C^*$-algebra $A$ is viewed as the algebra of continuous functions on a (uniquely determined) quantum space $X$, and we write $A=C(X)$. Quantum subspaces $Y \subset X$ correspond to surjective $*$-algebra maps $C(X) \rightarrow C(Y)$, and hence as well to ideals of $C(X)$. The quantum space $X$ is said to be classical (resp. non-classical) if the $C^*$-algebra $C(X)$ is commutative (resp. non-commutative).

\subsection{Classical spaces} The real and complex spheres are denoted respectively by $S^{n-1}_{\mathbb C}$
and $S^{n-1}_{\mathbb R}$, and their $C^*$-algebras of continuous functions always are endowed with their usual presentations:
$$C(S_{\mathbb C}^{n-1})= C^*\left(z_1, \ldots , z_n \ | \ \sum_{i=1}^nz_iz_i^*=1,  \ z_i z_j=z_jz_i, \ z_iz_j^*=z_j^*z_i, \ 1 \leq i,j,\leq n\right)$$
where $z_1, \ldots , z_n$ are the standard coordinate functions;
$$C(S_{\mathbb R}^{n-1})= C^*\left(x_1, \ldots , x_n \ | \ \sum_{i=1}^nx_i^2=1, \ x_i^*=x_i, \ x_ix_j=x_jx_i, \ 1 \leq i,j\leq n\right)$$
where $x_1, \ldots , x_n$ are the standard coordinate functions. 
The sphere $S_{\mathbb C}^0$ is, as usual, denoted $\mathbb T$.

The complex projective space $P^{n-1}_{\mathbb C}$ is the orbit space $S^{n-1}_{\mathbb C}/\mathbb T$ for the usual $\mathbb T$-action by multiplication. Of course $C(P^{n-1}_{\mathbb C})$ is isomorphic to a $C^*$-subalgebra of $C(S^{n-1}_{\mathbb C})$, via the natural identification $C(S^{n-1}_{\mathbb C}/\mathbb T) \simeq C(S^{n-1}_{\mathbb C})^\mathbb T$, the later $C^*$-algebra being the $C^*$-sub-algebra of $C(S^{n-1}_{\mathbb C})$ generated by the elements $z_iz_j^*$ (by the Stone-Weierstrass theorem). 
Moreover, the $C^*$-algebra $C(P^{n-1}_{\mathbb C})$ has the following presentation,  communicated to me by T. Banica.

\begin{lemma}\label{presproj}
 The $C^*$-algebra $C(P^{n-1}_{\mathbb C})$ has the presentation 
$$C(P_{\mathbb C}^{n-1}) \simeq C^*\left(p_{ij}, 1 \leq i,j \leq n \ | \ p=p^*=p^2, \ {\rm tr}(p)=1, \ p_{ij}p_{kl}=p_{kl}p_{ij}, \ 1 \leq i,j \leq n\right)$$
were $p$ denotes the matrix $(p_{ij})$, and where the element $p_{ij}$ corresponds to the element $z_iz_j^*$.
\end{lemma}

\begin{proof}
 Denote by $A$ the $C^*$-algebra on the right. It is straightforward to check that there exists a $*$-algebra map
\begin{align*}
  A &\longrightarrow C(P^{n-1}_{\mathbb C}) \\
p_{ij} &\longmapsto z_iz_j^* 
\end{align*}
which is surjective. To show the injectivity, it is enough, by Gelfand duality,  to show that the corresponding continuous map 
\begin{align*}
P^{n-1}_ {\mathbb C} = S^{n-1}_ {\mathbb C}/\mathbb T& \longrightarrow  \{ p \in M_n(\mathbb C), \ p=p^*=p^2, \ {\rm tr}(p)=1 \} \\
(z_1, \ldots , z_n) & \longmapsto (z_iz_j^*)
\end{align*}
is surjective, which follows from the standard reduction theory of projections.
\end{proof}

To conclude this section, we introduce a last piece of notation. The complex conjugation induces an order two automorphism of $C(S_{\mathbb C}^{n-1})$, that we denote $\tau$, with $\tau(z_i)=z_i^*$. This enables us to form the crossed product $C(S_{\mathbb C}^{n-1}) \rtimes \mathbb Z_2$, that we use intensively in the rest of the paper.

\section{Half-liberated spheres and their symmetric subspaces}\label{sec:sym}

Recall from the introduction that the half-liberated real sphere $S_{\mathbb R, *}^{n-1}$  \cite{bago} is the quantum space corresponding to the $C^*$-algebra
$$C(S_{\mathbb R, *}^{n-1})= C^*\left(v_1, \ldots , v_n \ | \ \sum_{i=1}^nv_i^2=1, \ v_i^*=v_i, \ v_iv_jv_k=v_kv_jv_i, \ 1 \leq i,j,k\leq n\right)$$

\subsection{$\mathbb Z_2$-grading on $C(S^{n-1}_{\mathbb R,*})$}

\begin{definition}
 The sign automorphism of $C(S^{n-1}_{\mathbb R,*})$, denoted $\nu$, is the automorphism defined by $\nu(v_i)=-v_i$, for any $i$.
\end{definition}


The sign automorphism defines a $\mathbb Z_2$-grading on the algebra $C(S^{n-1}_{\mathbb R,*})$:
$$C(S^{n-1}_{\mathbb R,*})=C(S^{n-1}_{\mathbb R,*})_0 \oplus C(S^{n-1}_{\mathbb R,*})_1$$
where $C(S^{n-1}_{\mathbb R,*})_0=\{a \in C(S^{n-1}_{\mathbb R,*})\ | \ \nu(a)=a\}$ and $C(S^{n-1}_{\mathbb R,*})_1=\{a \in C(S^{n-1}_{\mathbb R,*}) \ | \ \nu(a)=-a\}$.
Here of course $C(S^{n-1}_{\mathbb R,*})_0$ is the fixed point algebra for the $\mathbb Z_2$-action on $C(S^{n-1}_{\mathbb R,*})$ defined by $\nu$. It has the following description.

\begin{lemma}\label{zero}
 The $C^*$-subalgebra generated by the elements $v_iv_j$, $1 \leq i,j \leq n$, is commutative, and coincides with $C(S^{n-1}_{\mathbb R,*})_0$.
\end{lemma}

\begin{proof}
The commutativity  $C^*(v_iv_j)$, a fundamental and direct observation, is known \cite{bve,bago}.
 It is clear that $v_iv_j \in C(S^{n-1}_{\mathbb R,*})_0$ for any $i,j$, hence $C^*(v_iv_j)\subset C(S^{n-1}_{\mathbb R,*})_0$. Denote by $\mathcal O(S^{n-1}_{\mathbb R,*})$ the dense $*$-subalgebra of $C(S^{n-1}_{\mathbb R,*})$ generated by the elements $v_i$. It is clear that $\mathcal O(S^{n-1}_{\mathbb R,*})_0$ consists of the linear span of monomials of even length, hence is generated as an algebra by the elements $v_iv_j$. Therefore we get the announced result, since $\mathcal O(S^{n-1}_{\mathbb R,*})_0$ is dense in $C(S^{n-1}_{\mathbb R,*})_0$ (if $A$ is $C^*$-algebra acted on by a finite group and $\mathcal A \subset A$ is a dense $*$-subalgebra, then $\mathcal A^G$ is dense in $A^G$).
\end{proof}

\begin{definition}
A quantum subspace $X \subset S^{n-1}_{\mathbb R,*}$ is said to be symmetric if the corresponding ideal $I$
is $\mathbb Z_2$-graded, i.e. $$I=I_0\oplus I_1, \ {\rm with} \ I_0=I \cap C(S^{n-1}_{\mathbb R,*})_0 \ {\rm  and} \ I_1= I \cap C(S^{n-1}_{\mathbb R,*})_1$$ or in other words, if $\nu$ induces an automorphism of $C(S^{n-1}_{\mathbb R,*})/I$.
\end{definition}

\subsection{Faithful crossed product representation of $C(S^{n-1}_{\mathbb R,*})$}
We now describe a crossed product model for $C(S^{n-1}_{\mathbb R,*})$, using the crossed product $C(S_{\mathbb C}^{n-1}) \rtimes \mathbb Z_2$ associated to the conjugation action on $S^{n-1}_{\mathbb C}$, see the previous Section. A related construction was already considered in the quantum group setting in \cite{bdv}.

\begin{theorem}\label{faithcrossed}
There exists an injective $*$-algebra map
\begin{align*}
 \pi : C(S^{n-1}_{\mathbb R,*}) & \longrightarrow C(S_{\mathbb C}^{n-1}) \rtimes \mathbb Z_2 \\
v_i & \longmapsto z_i \otimes \tau
\end{align*}
\end{theorem}

\begin{proof}
It is straightforward to construct $\pi$, and this is left to the reader, and we have to show that $\pi$ is injective. First note that $C(S_{\mathbb C}^{n-1}) \rtimes \mathbb Z_2$ is $\mathbb Z_2$-graded as well, with 
grading defined by 
$$(C(S_{\mathbb C}^{n-1}) \rtimes \mathbb Z_2)_0=C(S_{\mathbb C}^{n-1}) \otimes 1, \ (C(S_{\mathbb C}^{n-1}) \rtimes \mathbb Z_2)_1=C(S_{\mathbb C}^{n-1}) \otimes \tau$$
and that $\pi$ preserves the respective $\mathbb Z_2$-gradings. Thus a standard argument shows that it is enough to show that the restriction of $\pi$ to $C(S^{n-1}_{\mathbb R,*})_0$ is injective. This is known, combining Lemma \ref{zero} and Theorem 3.3 in \cite{bago}, but for the sake of completeness, we provide a proof.
We have $\pi(v_iv_j)= z_iz_j^*\otimes 1$, and by Lemma \ref{presproj}, we see that there exists a $*$-algebra map 
$\Phi : C(P^{n-1}_{\mathbb C}) \rightarrow  C(S^{n-1}_{\mathbb R,*})$, $z_iz_j^* \mapsto v_iv_j$. We have
${\rm Im}(\Phi)=C(S^{n-1}_{\mathbb R,*})_0$ by Lemma \ref{zero}, and since $\pi\Phi={\rm id}\otimes 1$, we see that $\Phi$ is injective and induces an isomorphism $C(P^{n-1}_{\mathbb C}) \simeq  C(S^{n-1}_{\mathbb R,*})_0$. The restriction of $\pi$ to  $C(S^{n-1}_{\mathbb R,*})_0$ is $\Phi^{-1}\otimes 1$, hence $\pi$ is injective.
\end{proof}

For future use, we record the following fact, which is Theorem 3.2 in \cite{bago},  and was just reproved during the proof of Theorem \ref{faithcrossed}.

\begin{theorem}\label{future}
 There exists a $*$-algebra isomorphism \begin{align*}
                                        \Phi : C(P^{n-1}_{\mathbb C}) &\longrightarrow  C(S^{n-1}_{\mathbb R,*})_0 \\
 z_iz_j^* &\longmapsto v_iv_j \end{align*}
\end{theorem}

\subsection{Symmetric subspaces of half-liberated spheres} Before describing the symmetric subspaces of $S^{n-1}_{\mathbb R,*}$, we need a last ingredient.

\begin{definition}
 We denote by  $\gamma$ the linear endomorphism of $C(S^{n-1}_{\mathbb R,*})$ defined by $$\gamma(a)=\sum_{i=1}^nv_iav_i$$
\end{definition}
 
The main properties of $\gamma$ are summarized in the following lemma.
 
\begin{lemma}\label{gamma}
 The endomorphism $\gamma$ preserves the $\mathbb Z_2$-grading of $C(S^{n-1}_{\mathbb R,*})$, and induces a $*$-algebra automorphism of $C(S^{n-1}_{\mathbb R,*})_0$. Moreover the following diagram commutes 
\begin{equation*}
\begin{CD}
  C(P^{n-1}_{\mathbb C}) @>\Phi>>   C(S^{n-1}_{\mathbb R,*})_0 \\
@VV\tau V @VV\gamma V \\
 C(P^{n-1}_{\mathbb C}) @>\Phi>>   C(S^{n-1}_{\mathbb R,*})_0 
\end{CD}
\end{equation*}
where $\tau$ is the automorphism induced by complex conjugation. Hence there is a bijective correspondence between $\gamma$-stable ideals of $C(S^{n-1}_{\mathbb R,*})_0$ and conjugation stable closed subsets of $P^{n-1}_{\mathbb C}$.
\end{lemma}

\begin{proof} It is clear that $\gamma$ preserves the $\mathbb Z_2$-grading of $C(S^{n-1}_{\mathbb R,*})$, that $\gamma(1)=1$ and that $\gamma$ commutes with the involution. We have
$$\gamma(v_{i_1}v_{j_1} \cdots v_{i_m}v_{j_m}) =\sum_k v_k v_{i_1}v_{j_1} \cdots v_{i_m}v_{j_m} v_k= \sum_k v_{j_1}v_{i_1} \cdots v_{j_m}v_{i_m} v_kv_k= v_{j_1}v_{i_1} \cdots v_{j_m}v_{i_m}$$
This shows that the diagram commutes, and at the same time  that $\gamma$ preserves multiplication, and is an automorphism. The last assertion then follows immediately from the correspondence between conjugation closed subspaces of $P^{n-1}_{\mathbb C}$ and $\tau$-stable ideals of $C(P^{n-1}_{\mathbb C})$.
\end{proof}

The description of the $\mathbb Z_2$-graded ideals of $C(S^{n-1}_{\mathbb R,*})$ is then as follows.

\begin{theorem}
 We have a bijective correspondence 
\begin{align*}\{ \mathbb Z_2-{\rm graded} \ {\rm ideals} \  {\rm of} \ C(S^{n-1}_{\mathbb R,*})\}& \longleftrightarrow \{\gamma-{\rm stable} \ {\rm ideals} \ {\rm of} \  C(S^{n-1}_{\mathbb R,*})_0 \} \\
I & \longmapsto I_0=I \cap C(S^{n-1}_{\mathbb R,*})_0 \\
\langle J \rangle = J +   C(S^{n-1}_{\mathbb R,*})_1 J & \longleftarrow J 
\end{align*}
\end{theorem}

\begin{proof}
 For notational simplicity, we put $A=C(S^{n-1}_{\mathbb R,*})$. Let $I=I_0+I_1$ be a $\mathbb Z_2$-graded ideal of $A$.  It is clear that $I_0$ is  $\gamma$-stable since it is an ideal in $A$, thus the first map is well-defined, we call it $\mathcal F$.

Let $J \subset A_0$ be a $\gamma$-stable ideal. It is clear that  $J +   A_1 J$ is a left ideal, and, in order to show that it is as well a right ideal, it is enough to see that $A_1 J=JA_1$. This will follow from the following claim:
for $x \in A_0$, we have $v_ix = \gamma(x)v_i$ for any $i$. This is shown 
\begin{enumerate}
\item for elements of type $v_jv_k$, using half commutation,
\item  for polynomials in $v_jv_k$, by induction, 
\item and finally by density for any $x$. 
\end{enumerate}
From this, we have $v_i J= \gamma(J)v_i=Jv_i$, and  then $A_1J = \sum_i v_iA_0J= \sum_iv_i J=\sum_i Jv_i=\sum_i JA_0v_i = JA_1$, as required. Thus $J +   A_1 J$  is an ideal, and is $\mathbb Z_2$-graded by construction. The second map is well-defined, we call it $\mathcal G$.

Let $I=I_0+I_1$ be a $\mathbb Z_2$-graded ideal of $A$. Then
for $x \in I_1$, we have
$$x= \sum_i v_i v_i x \in A_1 I_0, \ x = \sum_i xv_i v_i \in I_0A_1$$
hence $I_1=A_1 I_0= I_0A_1$, and we have $\mathcal G \mathcal F(I)=I_0 + A_1I_0= I$.

Finally it is clear that $\mathcal F \mathcal G(J)=J$ for any $\gamma$-stable ideal $J \subset A_0$, and this concludes the proof.
\end{proof}

In terms of subspaces, we have the following immediate translation.

\begin{corollary}\label{symsub}
We have a bijection
$$\{ {\rm conjugation} \  {\rm stable} \ {\rm closed} \ {\rm subspaces} \ Y \subset P^{n-1}_{\mathbb C}\}
\longleftrightarrow
 \{{\rm symmetric} \  {\rm quantum} \ {\rm subspaces} \  X \subset S^{n-1}_{\mathbb R,*}\} $$
The bijection is as follows: let $Y \subset P^{n-1}_{\mathbb C}$ be a conjugation stable closed subspace, and let $\mathcal J_Y \subset C(P^{n-1}_{\mathbb C})$ be the ideal of functions vanishing on $Y$. Then the corresponding symmetric closed subspace $X$ of  $S^{n-1}_{\mathbb R,*}$ is defined by $C(X) = C( S^{n-1}_{\mathbb R,*})/ \langle\Phi(\mathcal J_Y)\rangle$.
\end{corollary}

The correspondence is  easy to use in practice, because $\Phi$ and $\Phi^{-1}$ are completely explicit, and $\Phi$
transforms the monomial $z_{i_1}\cdots z_{i_m} z_{j_1}^* \cdots z_{j_m}^*$ into the monomial $v_{i_1}v_{j_1} \cdots v_{i_m}v_{j_m}$. For example if 
 $f_1, \ldots, f_r$ are polynomials in $z_iz_j^*$ that generate the ideal $\mathcal J_Y$, then the corresponding quotient of $C(S^{n-1}_{\mathbb R,*})$ is $C(S^{n-1}_{\mathbb R,*})/(\Phi(f_1), \ldots , \Phi(f_r))$.

\section{General setup}\label{sec:general}

We now describe all the quantum subspaces  of $S^{n-1}_{\mathbb R,*}$. For this, it will be convenient to work in a more general framework.

We denote by $\mathbb T \rtimes \mathbb Z_2$ the semi-direct product associated to the conjugation action of 
$\mathbb Z_2$ on $\mathbb T$. This is a compact  group, isomorphic with the orthogonal group $O_2$, but the above description will be the more convenient.

\textsl{General setup.}
Let $Z$ be a compact space endowed with a continuous action of $\mathbb T \rtimes \mathbb Z_2$, such that the action is $\mathbb T$-free. The $\mathbb Z_2$-action on $Z$ corresponds, unless otherwise specified, to the $\mathbb Z_2$-action of the second factor of  $\mathbb T \rtimes \mathbb Z_2$, and the corresponding automorphism is denoted $\tau$.

We put $Z_{\mathbb R} = Z^{\mathbb Z_2}=\{z \in Z \ | \ \tau(z)=z\}$, $Z_{\rm reg}= Z \setminus \mathbb TZ_{\mathbb R}$,
and 
$$C(Z)^{\mathbb T} = \{f \in C(Z), \ f(\omega z)=f(z), \forall \omega \in \mathbb T, \forall z \in Z\}$$
$$C_{\mathbb T}(Z)=\{f \in C(Z), \ f(\omega z)=\omega f(z), \forall \omega \in \mathbb T, \forall z \in Z\}$$
The $\mathbb Z_2$-action on $Z$ enables us to form the crossed product
$C(Z)\rtimes \mathbb Z_2$.

\begin{definition} 
For   a compact space $Z$ endowed with a continuous $\mathbb T \rtimes \mathbb Z_2$-action, such that the action is $\mathbb T$-free, we put
$$C(Z_{\mathbb R,*})= \{f_0 \otimes 1 + f_1\otimes \tau, \ f_0 \in C(Z)^{\mathbb T}, \ f_1 \in C_{\mathbb T}(Z)\}\subset C(Z)\rtimes \mathbb Z_2$$
\end{definition}

\begin{proposition}
 $C(Z_{\mathbb R,*})$ a $C^*$-subalgebra of $C(Z)\rtimes \mathbb Z_2$, which is non-commutative if and only if $Z_{\rm reg} \not=\emptyset$. 
\end{proposition}

\begin{proof}
 The first assertion is a direct verification left to the reader, while the second one follows  from the forthcoming Proposition \ref{reptheory}.
\end{proof}

The following lemma links the present construction  to the half-liberated spheres.

\begin{lemma}
 Assume that there exist $f_1, \ldots , f_n \in C_{\mathbb T}(Z)$ such that
\begin{enumerate}
\item $C(Z)^{\mathbb T}=C^*(f_if_j^*, \ 1 \leq i,j \leq n)$,
\item $C_{\mathbb T}(Z)= f_1 C(Z)^{\mathbb T} + \cdots + f_nC(Z)^{\mathbb T}$.
\end{enumerate}
Then $C(Z_{\mathbb R,*})= C^*(f_1 \otimes \tau, \ldots, f_n\otimes \tau)$, and the elements $f_1\otimes \tau, \ldots , f_n\otimes \tau$ half commute.
\end{lemma}

\begin{proof}
Put $A= C^*(f_1 \otimes \tau, \ldots, f_n\otimes \tau)$. We have 
$$(f_i\otimes \tau) (f_j\otimes \tau)^*=(f_i\otimes \tau) (\tau(f_j^*)\otimes \tau)=
f_if_j^*\otimes 1\in A$$
 hence by (1) we have $C(Z)^{\mathbb T}\otimes 1 \subset A$. For $f \in C_{\mathbb T}(Z)$, we have
$f= \sum_ig_if_i$ for some $g_1, \ldots ,g_n \in C(Z)^{\mathbb T}$ by (2), hence
$$f \otimes \tau = \sum_i g_if_i \otimes \tau = \sum_i (g_i \otimes 1)(f_i\otimes \tau)\in A$$
and this shows that $C(Z_{\mathbb R,*})=A$. Moreover
$$(f_i\otimes \tau) (f_j\otimes \tau) (f_k\otimes \tau) = f_i\tau(f_j)f_k\otimes \tau= (f_k\otimes \tau) (f_j\otimes \tau) (f_i\otimes \tau)$$
which concludes the proof.
\end{proof}

\begin{example}
{\rm Consider the natural $\mathbb T \rtimes \mathbb Z_2$-action on $S^{n-1}_{\mathbb C}$, where the $\mathbb T$-action is by multiplication and the $\mathbb Z_2$-action is by conjugation. The $\mathbb T$-action is indeed free, we have $(S^{n-1}_{\mathbb C})_{\mathbb R}= S^{n-1}_{\mathbb R}$, and
$$(S^{n-1}_{\mathbb C})_{\rm reg}=:S^{n-1}_{\mathbb C, {\rm reg}} = S^{n-1}_{\mathbb C}\setminus \mathbb T S^{n-1}_{\mathbb R}
=\{g=(g_1, \ldots , g_n) \in S^{n-1}_{\mathbb C} \ | \ 
\exists i,j \ {\rm with}  \ g_i\overline{g_j}\not=g_j\overline{g_j}\}$$
The coordinate functions $z_1, \ldots , z_n \in C_{\mathbb T}(S^{n-1}_{\mathbb C})$ satisfy the conditions of the previous lemma, because $1=\sum_iz_iz_i^*$, and hence the image of the injective morphism
\begin{align*}
 \pi : C(S^{n-1}_{\mathbb R,*}) & \longrightarrow C(S_{\mathbb C}^{n-1}) \rtimes \mathbb Z_2 \\
v_i & \longmapsto z_i \otimes \tau
\end{align*}
of Theorem \ref{faithcrossed} 
is precisely $C((S^{n-1}_{\mathbb C})_{\mathbb R, *})$. Therefore we identify the two algebras.
} 
\end{example}

The description of the quantum subspaces of $Z_{\mathbb R, *}$, which  has $S^{n-1}_{\mathbb R, *}$ as a particular case, is as follows.

\begin{theorem}\label{closed1}
 There exists a bijection between the set of quantum subspaces $X \subset Z_{\mathbb R,*}$ and the set of pairs $(E,F)$ where
\begin{enumerate}
 \item $E \subset Z_{{\rm reg}}$ is $\mathbb T \rtimes \mathbb Z_2$-stable and $E= \overline{E} \cap Z_{{\rm reg}}$ (i.e. $E$ is closed in $Z_{\rm reg}$);
\item $F \subset Z_{\mathbb R}$ is closed and satisfies $\overline{E}\cap Z_{\mathbb R} \subset F$. 
\end{enumerate}
Moreover the quantum subspace $X$ is non-classical if and only if, in the corresponding pair $(E,F)$, we have $E \not=\emptyset$.
\end{theorem}

The proof is given in the next subsections.

\subsection{Representation theory of $C(Z_{\mathbb R,*})$}
We now provide the description of the irreducible representations of the $C^*$-algebra $C(Z_{\mathbb R, *}$), the main step towards the description of the subspaces of $Z_{\mathbb R, *}$. It is certainly possible to use general results on crossed products \cite{wil} to provide this description, but since everything can be done in a quite direct and elementary manner, we will proceed directly. 

The $C^*$-algebra $C(Z_{\mathbb R,*})$ is, by definition, a $C^*$-subalgebra of the crossed product $C^*$-algebra $C(Z)\rtimes \mathbb Z_2$. In particular, it can be seen as a $C^*$-subalgebra of $M_2(C(Z))$, and hence all its irreducible representations have dimension $\leq 2$, and  is a $2$-subhomogeneous $C^*$-algebra. The precise description of the irreducible representations of 
$C(Z_{\mathbb R, *})$ is given in the following result. 

\begin{proposition}\label{reptheory}
\begin{enumerate} \item Any $z \in Z$ defines a representation
\begin{align*}
 \theta_z : C(Z_{\mathbb R,*}) & \longrightarrow M_2(\mathbb C) \\
f_0 \otimes 1 + f_1\otimes \tau & \longmapsto 
\begin{pmatrix} 
 f_0(z) & f_1(z) \\
f_1(\tau(z)) & f_0(\tau(z))
\end{pmatrix}
\end{align*}
and any irreducible representation of $C(Z_{\mathbb R,*})$ is isomorphic to a sub-representation of some $\theta_z$, for some $z \in Z$. The representation $\theta_z$  is irreducible if and only if $z \in Z_{\rm reg}$. Moreover, for $z,x \in Z$, the representations $\theta_z$ and $\theta_x$ are isomorphic if and only if $(\mathbb T \rtimes \mathbb Z_2)z= (\mathbb T \rtimes \mathbb Z_2)x$.
\item Any $z \in Z_{\mathbb R}$ defines a one-dimensional representation
\begin{align*}
 \phi_z : C(Z_{\mathbb R,*}) & \longrightarrow \mathbb C\\
f_0 \otimes 1 + f_1\otimes \tau & \longmapsto 
f_0(z)+f_1(z)
\end{align*}
and any  one-dimensional representation arises in this way. Moreover for $z,y \in Z_{\mathbb R}$, we have $\phi_z=\phi_y \iff z=y$.
\item If $\pi$ is an irreducible representation of  $C(Z_{\mathbb R,*})$, then either $\pi \simeq \theta_z$ for some $z \in Z_{\rm reg}$ or $\pi=\phi_z$ for some $z \in Z_{\mathbb R}$ 
\end{enumerate}
\end{proposition}

\begin{proof}
It can be checked directly that $\theta_z$ defines a representation of $C(Z_{\mathbb R,*})$, or by using the standard embedding of the crossed product $C(Z)\rtimes \mathbb Z_2$ into $M_2(C(Z))$
\begin{align*}
 C(Z_{\mathbb R,*}) & \longrightarrow M_2(C(Z))\\
f_0 \otimes 1 + f_1\otimes \tau & \longmapsto 
\begin{pmatrix} 
 f_0 & f_1 \\
f_1\tau & f_0\tau
\end{pmatrix}
\end{align*}
composed with evaluation at $z \in Z$. Since any irreducible representation of $M_2(C(Z))$ is obtained by evaluation at an element $z \in Z$, we get that any irreducible representation of $C(Z_{\mathbb R,*})$ is isomorphic to a sub-representation of $\theta_z$ for some $z \in Z$, see e.g. \cite{dix}.

Now assume that $z \in \mathbb T Z_{\mathbb R}$: $z= \lambda y$
 for some $y \in Z_{\mathbb R}$. Then $\tau(z) = \tau(\lambda y)= \overline{\lambda}\tau(y)= \overline{\lambda} y=\overline{\lambda}^2z$, and for $f_0 \otimes 1 + f_1\otimes \tau \in C(Z_{\mathbb R, *})$, we have
$$\theta_z(f_0 \otimes 1 + f_1\otimes \tau)= \begin{pmatrix} 
 f_0(z) & f_1(z) \\
f_1(\tau(z)) & f_0(\tau(z))
\end{pmatrix} = \begin{pmatrix} 
 f_0(z) & f_1(z) \\
\overline{\lambda}^2f_1(z) & f_0(z)
\end{pmatrix}$$
This implies that $\theta_z(C(Z_{\mathbb R, *}))$ is abelian, and hence $\theta_z$ is not irreducible.

Assume that $z \in Z_{\rm reg}$. To show that $\theta_z$ is irreducible, it is enough to show that there exists
$f_0 \in C(Z)^{\mathbb T}$ and $f_1 \in C_{\mathbb T}(Z)$ such that
$$f_0(z) \not=f_0(\tau(z)), \ f_1(z) \not=0$$
Indeed, we then will have that $\theta_z(f_0 \otimes 1)$ and $\theta_z(f_1 \otimes \tau)$ do not commute, and hence
 $\theta_z(C(Z_{\mathbb R, *}))$ is a non-commutative $C^*$-subalgebra of $M_2(\mathbb C)$, so both algebras are equal and $\theta_z$ is irreducible. 
 
We have, since $z \in Z_{\rm reg}$, $\mathbb T z \not = \mathbb T \tau(z)$. Otherwise $z=\lambda \tau(z)$ for some $\lambda \in \mathbb T$, and for $\mu \in \mathbb T$ such that $\mu^2=\overline{\lambda}$, we have $z= \overline{\mu} \mu z$, with $\mu z \in Z_{\mathbb R}$, a contradiction. Hence by Urysohn's Lemma and the fact that $X/\mathbb T$ is Hausdorff,  there exists $f_0 \in C(Z)^{\mathbb  T}\simeq C(Z/\mathbb T)$ such that $f_0(z)\not=f_0(\tau(z))$, as needed.
Finally since the $\mathbb T$-action is free, there exists a (continuous) map $f : \mathbb T z \rightarrow \mathbb C$ such that $f(\lambda z)=\lambda$ for any $\lambda \in \mathbb T$, that we extend to to a continuous function $f$ on $Z$ (Tietze's extension theorem).
Now let 
$f_1\in C(Z)$ be defined by
$$f_1(y) =\int_{\mathbb T} \lambda^{-1}f(\lambda y){\rm d}\lambda$$  
We have $f_1 \in C_{\mathbb T}(Z)$ and $f_1(z)=1$, as needed, and we conclude that  $\theta_z$ is irreducible.

A finite-dimensional representation is determined by its character, and the character of the representation $\theta_z$ is given by $\chi_z(f_0 \otimes 1 + f_1\otimes \tau)=f_0(z)+f_0(\tau(z))$.

Let $z,x \in Z$ be such that $(\mathbb T \rtimes \mathbb Z_2)z\not= (\mathbb T \rtimes \mathbb Z_2)x$. 
Then there exists $f_0 \in C(Z)^{\mathbb T \rtimes \mathbb Z_2}$ such that $f_0(z)=1=f_0(\tau(z))$ and $f_0(x)=0=f_0(\tau(x))$. We have $f_0 \in C(Z)^{\mathbb T}$ and $\chi_z(f)=2$, $\chi_x(f)=0$, hence the representations $\theta_z$ and $\theta_x$ are not isomorphic.

If $(\mathbb T \rtimes \mathbb Z_2)z= (\mathbb T \rtimes \mathbb Z_2)x$, then either $z= \lambda x$ or $z=\lambda \tau(x)$ for some $\lambda \in \mathbb T$. From this we see that $f(z)+f(\tau(z))=f(x)+f(\tau(x))$ for any $f \in C(Z)^{\mathbb  T}$, and hence
$\chi_z=\chi_x$, which shows that $\theta_z$ and $\theta_x$ are isomorphic, and concludes the proof of (1).

For $z \in Z_{\mathbb R}$, it is a direct verification to check that $\phi_z$ above defines a $*$-algebra map
 $C(Z_{\mathbb R,*})  \rightarrow \mathbb C$. Now let $\psi : C(Z_{\mathbb R,*})  \rightarrow \mathbb C$ be a $*$-algebra map. The representation defined by $\psi$ is isomorphic to a sub-representation of
$\theta_z$ for some $z \in Z$, and with  $z \not \in Z_{\reg}$ because $\theta_z$ is not irreducible.  Hence $z=\lambda y$ for $y \in Z_{\mathbb R}$. We then have, for $f_0 \otimes 1 + f_1\otimes \tau \in C(Z_{\mathbb R, *})$,
$$\theta_z(f_0 \otimes 1 + f_1\otimes \tau)= \begin{pmatrix}
                                              f_{0}(y) & \lambda f_{1}(y) \\
\overline{\lambda} f_1(y) & f_0(y) 
                                             \end{pmatrix}$$
and from this we see that the lines generated by $(1, \overline{\lambda})$ and $(1, -\overline{\lambda})$ both are stable under $C(Z_{\mathbb R, *})$, so that $\theta_z \simeq \phi_y \oplus \phi_{-y}$ (this can also be seen using characters), and finally $\psi = \phi_{y}$ or $\psi=\phi_{-y}$. 

Let $y,z \in Z_{\mathbb R}$ be such that $\phi_y=\phi_z$. Then for any $f_0 \in C(Z)^{\mathbb T}$, we have $f_0(y)=f_0(z)$, hence $\mathbb T y = \mathbb T z$, hence $y=\pm z$. As before, there exists $f_1 \in C_{\mathbb T}(Z)$ such that $f_1(\lambda z)=\lambda$ for any $\lambda \in \mathbb T$. Then $\phi_z(f_1 \otimes \tau) =f_1(\tau(z))=f_1(z)=1=\pm f_1(y) = \pm \phi_y(f_1 \otimes \tau)$, hence $z=y$.

This proves (2), and (3) follows from the combination of (1) and (2).
\end{proof}

\subsection{Closed subspaces of $\widehat{C(Z_{\mathbb R,*})}$}
We now discuss $\widehat{C(Z_{\mathbb R,*})}$, the spectrum of $C(Z_{\mathbb R,*})$, endowed with its usual topology,
see \cite{dix} (since all the irreducible representations of $A$ are finite-dimensional, we know \cite{dix} that the topological spaces $\widehat{A}$ and ${\rm Prim}(A)$ are canonically homeomorphic).

For $E \subset Z_{\rm reg}$  and  $F \subset Z_{\mathbb R}$, we put 
$$M(E,F)= \left(\{\theta_z, \ z \in E\} \cup \{\phi_{z}, z \in F\}\right)/\!\sim \ \ \subset \ \widehat{C(Z_{\mathbb R,*})} $$
where of course $\sim$ means that we identify isomorphic representations. 
The previous proposition ensures that any subset of $\widehat{C(Z_{\mathbb R,*})}$ is of the form $M(E,F)$ for $E \subset Z_{\rm reg}$, a  $\mathbb T \rtimes \mathbb Z_2$-stable subspace, and  $F \subset Z_{\mathbb R}$. 
The next result describes the closed subsets of  $\widehat{C(Z_{\mathbb R,*})}$.

\begin{proposition}\label{closed}
 Let $E \subset Z_{\rm reg}$ be a $\mathbb T \rtimes \mathbb Z_2$-stable subspace, and let $F \subset Z_{\mathbb R}$. Then we have
$$\overline{M(E,F)} = M(\overline{E} \cap Z_{\rm reg}, \overline{F} \cup (\overline{E}  \cap Z_{\mathbb R}))$$
In particular there exists a bijection between closed subsets of $\widehat{C(Z_{\mathbb R,*})}$ and pairs $(E,F)$
with \begin{enumerate}
 \item $E \subset Z_{{\rm reg}}$ is $\mathbb T \rtimes \mathbb Z_2$-stable and $E= \overline{E} \cap Z_{{\rm reg}}$ ;
\item $F \subset Z_{\mathbb R}$ is closed and satisfies $\overline{E}\cap Z_{\mathbb R} \subset F$. 
\end{enumerate}
\end{proposition}

\begin{proof}
 Put $A=C(Z_{\mathbb R,*})$. 
For $S \subset \widehat{A}$, we have
$$\overline{S} = \{ \pi \in \widehat{A} \ | \ \bigcap_{\rho \in S}{\rm Ker}(\rho) \subset {\rm Ker}(\pi)\}$$
For $E,F$ as in the statement of the proposition, we have
$$\overline{M(E,F)} = \overline{M(E,\emptyset) \cup M(\emptyset, F)} = \overline{M(E,\emptyset)} \cup \overline{M(\emptyset,F)}$$
Hence we can study the two pieces separately. The bijective map
\begin{align*}
 Z_{\mathbb R} &\longrightarrow \widehat{A}_1 \\
z & \longmapsto \phi_z
\end{align*}
where $\widehat{A}_1$ consists of the set of $1$-dimensional representations, is clearly continuous, and since $Z_{\mathbb R}$ and $\widehat{A}_1$ both are compact, this is an homeomorphism. Hence it sends the closure of a subset in $Z_{\mathbb R}$ to the closure of the image in $\widehat{A}_1$, and hence to the closure of the image in $\widehat{A}$, since $\widehat{A}_1$ is closed in $\widehat{A}$. Thus $\overline{M(\emptyset,F)}= M(\emptyset , \overline{F})$. Consider now the following two claims:
$${\rm For} \ y \in Z_{\mathbb R}, \ \bigcap_{z \in E} {\rm Ker}(\theta_{z}) \subset {\rm Ker}(\phi_y) \iff y \in \overline{E} \cap Z_{\mathbb R} \quad (*)$$
$${\rm For} \ y \in Z_{\rm reg}, \ \bigcap_{z \in E} {\rm Ker}(\theta_{z}) \subset {\rm Ker}(\theta_y) \iff y \in \overline{E} \cap Z_{\rm reg}\quad (**)$$
Once these claims are proved, we will indeed have
$$\overline{M(E,\emptyset)} = M(\overline{E} \cap Z_{\rm reg}, \overline{E}  \cap Z_{\mathbb R})$$
as required. We begin with ($*$). Let $y \in Z_{\mathbb R}$. Assume first that $y \in \overline{E} \cap Z_{\mathbb R}$. Let $f_0 \otimes 1 + f_1 \otimes \tau \in  \cap_{z \in E} {\rm Ker}(\theta_{z})$. Then 
$f_0$ and $f_1$ vanish on $E$, and hence on $\overline{E}$, and $f_0 \otimes 1 + f_1 \otimes \tau \in {\rm Ker}(\phi_y)$. Thus $\cap_{z \in E} {\rm Ker}(\theta_{z}) \subset {\rm Ker}(\phi_y)$. Conversely assume that $y \not  \in \overline{E} \cap Z_{\mathbb R}$. Then $\mathbb Ty \cap \overline{E}=\emptyset$ since $E$ is $\mathbb T$-stable and there exists $f_0 \in C(Z)^{\mathbb T}$ such that $f_0(y)=1$ and $f_0(\overline{E})=0$. We then have 
$f_0 \otimes 1 \in  \cap_{z \in E} {\rm Ker}(\theta_{z})$ while $f_0 \otimes 1 \not \in {\rm Ker}(\phi_y)$, and ($*$) is proved.

Now let $y \in Z_{\rm reg}$. Assume first that $y \in \overline{E} \cap Z_{\rm reg}$. Let $f_0 \otimes 1 + f_1 \otimes \tau \in  \cap_{z \in E} {\rm Ker}(\theta_{z})$. Then 
$f_0$ and $f_1$ vanish on $E$, and hence on $\overline{E}$, and $f_0 \otimes 1 + f_1 \otimes \tau \in {\rm Ker}(\theta_y)$. Thus $\cap_{z \in E} {\rm Ker}(\theta_{z}) \subset {\rm Ker}(\theta_y)$. Conversely assume that $y \not  \in \overline{E} \cap Z_{\rm reg}$. Then $(\mathbb T\rtimes \mathbb Z_2)y \cap \overline{E}=\emptyset$ since $E$ is $\mathbb T \rtimes \mathbb Z_2$-stable and there exists $f_0 \in C(Z)^{\mathbb T\rtimes \mathbb Z_2}$ such that $f_0(y)=1$ and $f_0(\overline{E})=0$.  We then have 
$f_0 \otimes 1 \in  \cap_{z \in E} {\rm Ker}(\theta_{z})$ while $f_0 \otimes 1 \not \in {\rm Ker}(\theta_y)$, and ($**$) is proved.

For $E,E' \subset Z_{\rm reg}$,   and $F,F' \subset Z_{\mathbb R}$, then by Proposition \ref{reptheory} we have $\mathcal M(E,F)=\mathcal M(E',F')\Rightarrow$ $F=F'$ and $E=E'$ if $E$ and $E'$ are $\mathbb T \rtimes \mathbb Z_2$-stable, hence  the last assertion follows from the first one.
\end{proof}

We are now ready to prove the Theorem \ref{closed1}.

\begin{proof}[Proof of Theorem \ref{closed1}]
This follows from Proposition \ref{closed} and the standard correspondence between quotients of a $C^*$-algebra and closed subsets of its spectrum \cite{dix}. More precisely to a pair $(E,F)$ as in the statement is associated the $C^*$-algebra
$C(Z_{\mathbb R, *})/(\cap_{\pi \in M(E,F)}{\rm Ker}(\pi))$, and conversely, to a quotient $C^*$-algebra map $q : C(Z_{\mathbb R, *}) \rightarrow B$, is associated the pair $(E,F)$ such that 
$\{\pi \in \widehat{C(Z_{\mathbb R, *})} \ | \  {\rm Ker}(q) \subset {\rm Ker}(\pi)\}=M(E,F)$, i.e.
$$E= \{z \in Z_{\rm reg} \ | \  {\rm Ker}(q) \subset {\rm Ker}(\theta_z)\}, \quad F= \{z \in Z_{\mathbb R} \ | \  {\rm Ker}(q) \subset {\rm Ker}(\phi_z)\}$$
The $C^*$-algebra $B$ is non-commutative if and only if it has an irreducible representation of dimension $>1$, if and only if the corresponding $E$ is non-empty. 
\end{proof}

\begin{remark}{\rm 
 It follows from the above considerations that plenty of intermediate quantum subspaces $S_{\mathbb R}^{n-1} \subset X \subset S_{\mathbb R, *}^{n-1}$ exist for $n \geq 2$. Indeed, for any $m \geq 1$, there exists $S_{\mathbb R}^{n-1} \subset X \subset S_{\mathbb R, *}^{n-1}$ such that $C(X)$ has precisely $m$ isomorphism classes of irreducible representations of dimension $2$. This follows from the fact that the orbit space $S_{\mathbb C, {\rm reg}}^{n-1}/(\mathbb T \rtimes \mathbb Z_2)$ is infinite, and hence if we pick $E \subset S_{\mathbb C, {\rm reg}}^{n-1}$ (closed and $\mathbb T \rtimes \mathbb Z_2$-stable) such that $E/(\mathbb T \rtimes \mathbb Z_2)$ has $m$ elements, the quantum space corresponding to the pair $(E, S_{\mathbb R}^{n-1})$ satisfies to the above property.
}
\end{remark}


\subsection{Symmetric subspaces}
It is now natural to wonder about the link between the description of the symmetric subspaces of $S^{n-1}_{\mathbb R,*}$ in Section 3 and the one that should arise from Theorem \ref{closed1}. We first note that there exists also a notion of symmetric subspace in the general framework.
Indeed, define an automorphism $\sigma$ of $C(Z_{\mathbb R,*})$ by $\sigma(f_0 \otimes 1 + f_1\otimes \tau)=f_0 \otimes 1 - f_1\otimes \tau$. We say that a subspace $X \subset Z_{\mathbb R,*}$ is symmetric if $\sigma(X)=X$, or in other words, if $\sigma$ induces an automorphism on the corresponding quotient of $C(Z_{\mathbb R,*})$.  When $Z=S^{n-1}_{\mathbb C}$, the automorphism $\sigma$ is the sign automorphism $\nu$ of Section 3.

 \begin{theorem}\label{closed2}
 There exists a bijection 
$$\{{\rm symmetric} \  {\rm quantum} \ {\rm subspaces} \  X \subset Z_{\mathbb R,*}\} \longleftrightarrow 
\{ \mathbb T\rtimes \mathbb Z_2{-\rm stable} \ {\rm closed} \ {\rm subspaces} \ Y \subset  Z\}$$ 
Moreover the subspace $X\subset  Z_{\mathbb R,*}$ is non-classical if and only if, for the  the corresponding $Y \subset  Z$, we have $Y_{\rm reg} =Y \cap Z_{\rm reg}\not=\emptyset$.
\end{theorem}

\begin{proof}
We retain the notation of Proposition \ref{reptheory}.
It is straightforward to check that for $z \in Z_{\rm reg}$, we have 
$\theta_z\sigma=\theta_{-z}\simeq \theta_{z}$ and that for $z \in Z_{\mathbb R}$, we have $\phi_z\sigma=\phi_{-z}$.

Now if $X \subset Z_{\mathbb R, *}$ is a quantum subspace, then saying that $X$ is symmetric precisely means that for
the corresponding pair $(E,F)$, we have that  the corresponding ideal $\cap_{\pi \in M(E,F)}{\rm Ker}(\pi)$ is $\sigma$-stable. We have
\begin{align*}\sigma(\cap_{\pi \in M(E,F)}&{\rm Ker}(\pi)) =\sigma( (\cap_{z \in E}{\rm Ker}(\theta_z)) \cap \sigma(\cap_{z \in F}{\rm Ker}(\phi_z)) \\
= &\left(\cap_{z \in E}{\rm Ker}(\theta_z\sigma^{-1})\right) \cap \left(\cap_{z \in E}{\rm Ker}(\phi_z\sigma^{-1})\right) 
=\left(\cap_{z \in E}{\rm Ker}(\theta_z)\right) \cap \left(\cap_{z \in F}{\rm Ker}(\phi_{-z})\right)\\
=& \cap_{\pi \in M(E,-F)}{\rm Ker}(\pi)\end{align*}
Therefore  $X \subset Z_{\mathbb R, *}$ is symmetric if and only if for the corresponding pair $(E,F)$, we have $F=-F$. To such a pair $(E,F)$,  we associate the closed $\mathbb T\rtimes \mathbb Z_2$-stable  subset $Y=E \cup \mathbb T F$.
Conversely, if $Y$ is $\mathbb T\rtimes \mathbb Z_2$-stable, then $(Y_{\rm reg}, Y_{\mathbb R})$ is a pair as above.
The two maps are inverse bijections.
\end{proof}

Of course the above Theorem reproves the first part of Corollary \ref{symsub}, since  closed $\mathbb T\rtimes \mathbb Z_2$-stable subspaces of $S^{n-1}_{\mathbb C}$ correspond to conjugation stable closed subspaces of $P^{n-1}_{\mathbb C}$. 

To finish, we provide another explicit description of the bijection in the proof of Theorem \ref{closed2}. If $Y \subset Z$ is a closed $\mathbb T\rtimes \mathbb Z_2{-\rm stable}$ subspace, then we may form the $C^*$-algebra $C(Y_{\mathbb R, *})$ as before, and the restriction of functions yields a $*$-algebra map $C(Z_{\mathbb R, *}) \rightarrow C(Y_{\mathbb R, *})$, which is easily seen to be surjective thanks to the standard extension theorems.
We thus get a quantum subspace $Y_{\mathbb R, *} \subset Z_{\mathbb R, *}$.

\begin{theorem}
 The map 
\begin{align*}
 \{ \mathbb T\rtimes \mathbb Z_2{-\rm stable} \ {\rm closed} \ {\rm subsets} \ Y \subset  Z\} & \longrightarrow
\{{\rm symmetric} \  {\rm quantum} \ {\rm subspaces} \  X \subset Z_{\mathbb R,*}\} \\
Y & \longmapsto Y_{\mathbb R, *}
\end{align*}
is a bijection.
\end{theorem}

\begin{proof}
 Let $q : C(Z_{\mathbb R, *}) \rightarrow C(Y_{\mathbb R, *})$ the surjective $*$-algebra map associated to the restriction of functions. It is a direct verification to check that
$${\rm Ker}(q)= (\cap_{y \in Y_{\rm reg}} {\rm Ker}(\theta_y)) \cap (\cap_{y \in Y_{\mathbb R}} {\rm Ker}(\phi_y))=
\bigcap_{\pi \in M(Y_{\rm reg},Y_{\mathbb R})}{\rm Ker}(\pi)$$
Hence, in terms of pairs $(E,F)$ corresponding to closed subsets of $\widehat{C(Z_{\mathbb R,*})},$ our map associates the pair $(Y_{\rm reg},Y_{\mathbb R})$ to $Y$. This map is, as already discussed, a bijection.
\end{proof}

\end{document}